\documentclass{amsart}

\usepackage{amsthm, amsmath, amssymb}
\usepackage[nobysame]{amsrefs}
\usepackage[margin=1in]{geometry}
\usepackage{enumitem}
\usepackage{setspace}

\SetLabelAlign{CenterWithParen}{\hfil{\upshape(\makebox[1.0em]{#1})}\hfil}

\newtheorem*{hyp*}{Hypothesis \((\ast )\)}
\newtheorem{thm}{Theorem}[section]

\newtheorem{lem}[thm]{Lemma}

\newcommand{\irr}[1]{\text{Irr}(#1)}
\newcommand{\cod}[1]{\text{cod}(#1)}
\newcommand{\cd}[1]{\text{cd}(#1)}
\renewcommand{\c}[1]{{c}(#1)}
\renewcommand{\ker}[1]{\text{ker}(#1)}

\begin{document}

\title{\(p\)-groups with exactly four codegrees} 

\author{Sarah Croome}
\address{%
Department of Mathematical Sciences\\
Kent State University\\
Kent, OH 44242}

\email{scroome@kent.edu}

\author{Mark L. Lewis}
\address{Department of Mathematical Sciences\\
Kent State University\\
Kent, OH 44242}
\email{lewis@math.kent.edu}


\begin{abstract} Let \(G\) be a \(p\)-group and let \(\chi\) be an irreducible character of \(G\). The codegree of \(\chi\) is given by \(|G:\text{ker}(\chi)|/\chi(1)\). Du and Lewis have shown that a \(p\)-group with exactly three codegrees has nilpotence class at most 2. Here we investigate \(p\)-groups with exactly four codegrees. If, in addition to having exactly four codegrees, \(G\) has two irreducible character degrees, \(G\) has largest irreducible character degree \(p^2\), \(|G:G'|=p^2\), or \(G\) has coclass at most 3, then \(G\) has nilpotence class at most 4. In the case of coclass at most 3, the order of \(G\) is bounded by \(p^7\). With an additional hypothesis we can extend this result to \(p\)-groups with four codegrees and coclass at most 7. In this case the order of \(G\) is bounded by \(p^{11}\). \end{abstract}

\maketitle

\section{Introduction}
In this paper all groups are finite \(p\)-groups.  For an irreducible character \(\chi\) of a group \(G\), the codegree of \(\chi\) is defined as \(\cod{\chi}=|G:\ker{\chi}|/\chi(1)\). The set of irreducible characters of \(G\) is denoted \(\irr{G}\), the set of codegrees of the irreducible characters of \(G\) is denoted \(\cod{G}\), and the nilpotence class of \(G\) is \(c(G)\). In \cite{codandnil}, Du and Lewis showed that a \(p\)-group with exactly two codegrees is elementary abelian, and a \(p\)-group with exactly three codegrees has nilpotence class at most 2. Our original purpose was to find a sharp bound for the nilpotence class of groups with exactly four codegrees. Ideally, this bound would be a simple constant. At this time, our best result for arbitrary \(p\)-groups depends on the largest codegree, but in several more specific cases we have attained the predicted bound of nilpotence class at most 4.

\begin{thm}
\label{tiered class4}
Let \(G\) be a finite \(p\)-group such that \(\cod{G}=\{1,p,p^b,p^a\}\) where \(2\le b<a\). If any of the following hold, then \(G\) has nilpotence class at most 4:
\begin{enumerate} [label=(\roman*),font=\upshape]
\item \label{class4_cd2} \(|\cd{G}|=2\), 
\item \label{class4_cd3}  \(\cd{G}=\{1,p,p^2\}\),
\item \label{2gen_class4} \(|G:G'|=p^2\).
\end{enumerate}
\end{thm}

The coclass of a \(p\)-group \(G\) with nilpotence class \(c\) is given by \(\log_p{|G|}-c\). If \(G\) has four codegrees and coclass at most \(3\), then \(G\) has nilpotence class at most 4, bounding the order of \(G\).

\begin{thm} \label{class4_coclasses}
Let \(G\) be a \(p\)-group such that \(\cod{G}=\{1,p,p^b,p^a\}\), where \(2\le b<a\). If \(G\) has coclass at most 3, then \(G\) has nilpotence class at most \(4\), and \(|G|\le p^7\).
\end{thm}

With the following additional hypothesis, we can extend the result of Theorem \ref{class4_coclasses} to \(p\)-groups with coclass at most \(7\). 

\begin{hyp*} If \(G\) is a \(p\)-group with nilpotence class \(n\) such that \(|G|\ge p^{2n}\), then \(|Z_2(G)|\ne p^2\).
\end{hyp*}

\begin{thm}
\label{star_class4_coclasses}
Let \(G\) be a finite \(p\)-group with \(\cod{G}=\{1,p,p^b,p^a\}\) where \(2\le b<a\). If \(G\) has coclass at most 7, and \(G\) and all of its quotients satisfy Hypothesis \((\ast)\), then the nilpotence class of \(G\) is at most 4 and \(|G|\le p^{11}\).
\end{thm}

In \cite{codandnil}, Du and Lewis were able to bound \(c(G)\) in terms of the largest member of \(\cod{G}\). They showed that if \(p^a\) (where \(a>1\)) is the largest codegree of \(G\),  then \(c(G)\le 2a-2\), and in some specific cases, \(c(G)\le 2a-3\). When \(|\cod{G}|=4\), we can improve this bound. 

\begin{thm}
\label{class_aplus1}
If \(G\) is a finite \(p\)-group with \(\cod{G}=\{1,p,p^b,p^a\}\) where \(2\le b< a\), then \(c(G)\le a+1\).
\end{thm}

This bound can be improved slightly when the two largest codegrees are consecutive powers of \(p\) and the group does not have \(p^2\) as  a codegree.

\begin{thm}
\label{class_a}
If \(G\) is a finite \(p\)-group such that \(\cod{G}=\{1,p,p^{a-1},p^a\}\) for \(a\ge 4\), then \(\c{G}\le a\).  
\end{thm}

We expect this work to appear as part of the first author's Ph.D. dissertation at Kent State University.

\section{Bounding nilpotence class}

Our first lemma relates the codegree of a faithful irreducible character to the order of the group.

\begin{lem}\cite[Lemma 3.1]{codandnil} \label{codorder}
Let \(G\) be a group and suppose that \(\chi\in\irr{G}\) is faithful. Then \(|G|=\chi(1)\cod{\chi}<\cod{\chi}^2\).
\end{lem}

Lemma \ref{nonlinear} can be inferred from \cite{codandnil}.

\begin{lem}
\label{nonlinear}
Let \(G\) be a \(p\)-group with \(p^2\notin\cod{G}\). If \(\chi\) is an irreducible character of \(G\) such that \(\cod{\chi}>p\), then \( \chi\) is non-linear. 
\end{lem}

\begin{proof} 
Let \(\chi\) be an irreducible character of \(G\) such that \(\cod{\chi}=p^a\) for some \(a>2\), and suppose that \(\chi\) is linear.  Then \(\ker{\chi}\ge G' \), so \(G/\ker{\chi}\) is abelian. Since \(\chi\) is a faithful irreducible character of  \(G/\ker{\chi}\), this quotient  must be cyclic, and \(|G/\ker{\chi}|=\chi(1)\cod{\chi}=p^a>p^2\). By Corollary 2.5 of \cite{codandnil}, \(G/G'\) is elementary abelian, which is impossible since \(\ker{\chi}\ge G'\) and \(G/\ker{\chi}\) is cyclic with order greater than \(p^2\). 
\end{proof}

The following useful lemma is a consequence of Proposition 2.5 of \cite{PPO1} and It\^o's Theorem \cite[Theorem 6.15]{thebook}.

\begin{lem}
\label{faithfulp}
If a finite \(p\)-group \(G\) has a faithful irreducible character of degree \(p\), then \(G\) has a normal abelian subgroup of index \(p\) and \(\cd{G}=\{1,p\}\).
\end{lem}

 A group with \(\cd{G}=\{1,p\}\) must satisfy one of two conditions first given by Isaacs and Passman in \cite{characterization1} as Theorem C 4.8. We state this theorem below as it appears in \cite{PPO1}.
 
 \begin{lem}\cite[Theorem 22.5]{PPO1}
 \label{cd 1p}
 If \(G\) is a nonabelian \(p\)-group with \(\cd{G}=\{1,p\}\), then one and only one of the following holds:
 \begin{enumerate} [label=(\roman*),font=\upshape]
 \item \(G\) has an abelian subgroup of index \(p\),
 \item \(G/Z(G)\) is of order \(p^3\) and exponent \(p\). 
 \end{enumerate}
 \end{lem}

Frequently, we will show through Lemma \ref{faithfulp} or Lemma \ref{cd 1p} that a group has an abelian subgroup of index \(p\), and then use the following lemma to obtain a contradiction about the possible orders of the group, the derived subgroup, and the center.

\begin{lem}\cite[Lemma 1.1]{PPO1}
\label{1stlem} Let \(A\) be an abelian subgroup of index \(p\) of a nonabelian \(p\)-group \(G\). Then \(|G|=p|G'||Z(G)|\).
\end{lem}

The proofs of the main theorems will make use of minimal counter examples to reach a contradiction. The next lemma will allow us to shorten these proofs.

\begin{lem}
\label{minimalfaithful}
Let \(G\) be a finite \(p\)-group such that \(\cod{G}=\{1,p,p^b,p^a\}\), where \(2\le b<a\). If \(|G|\) is minimal such that \(G\) has nilpotence class \(n\ge 3\), then \(G\) has a faithful irreducible character.\end{lem}

\begin{proof}
Let \(\cod{G}=\{1,p,p^b,p^a\}\) and let \(|G|\) be minimal such that \(G\) has class \(n\). Let \(K\) be the kernel of an irreducible character of \(G\). Either \(\cod{G/K}\subsetneq \cod{G}\), in which case \(c(G/K)\le 2\), or \(\cod{G/K}=\cod{G}\), and since \(|G|\) is minimal, \(c(G/K)\le n-1\). Hence, there is a central series of \(G\) with the \(n-1\) term contained in \(K\) for all kernels of irreducible characters of \(G\), and since the intersection of these kernels is trivial, \(G\) has class at most \(n-1\), a contradiction. Thus \(G\) must have a faithful irreducible character.\end{proof}

The upper bound on the nilpotence class of \(G\) when \(p^a=\text{max}(\cod{G})\) is given as \(2a-2\) or \(2a-3\) in \cite{codandnil}. Restricting \(|\cod{G}|\) to 4 yields the stronger bound of Theorem \ref{class_aplus1}.


\begin{proof}[Proof of Theorem \ref{class_aplus1}]
Let \( G\) be a minimal counterexample and notice that \(c(G)=a+2\), for otherwise \(c(G)\ge a+3\) implies \(c(G/Z)\ge a+2\ge 5 \), and by Theorem 1.2 of \cite{codandnil}, \(\cod{G/Z}=\{1,p,p^b,p^a\}\), contradicting that \(|G|\) is minimal. By Lemma \ref{minimalfaithful}, \(G\) has a faithful character of codegree at most \(p^a\) and by Lemma  \ref{codorder}, \(|G|\le p^{2a-1}\). Using Theorem 1.2 of \cite{codandnil} and that \(c(G/Z_{a-1})=3\), we  have \(\cod{G/Z_{a-1}}=\cod{G}\). Let \(\chi \) be an irreducible character of \(G/Z_{a-1}\) such that \(\cod{\chi}=p^a\). Then \( p^{a}\le \chi(1)p^a = |G:\ker{\chi}|\le |G:Z_{a-1}|\le p^a \) so \(\chi\) is a faithful linear character of \(G/Z_{a-1}\), which is impossible since \(G/Z_{a-1}\) has class 3. Hence \(c(G)\le a+1\).
\end{proof}

Theorem \ref{class_a} further improves the bound in Theorem \ref{class_aplus1} when the two largest codegrees are consecutive powers of \(p\) and \(p^2\notin \cod{G}\). 


\begin{proof}[Proof of Theorem \ref{class_a}]
Let \(G\) be a minimal counterexample. Then \(\c{G}=a+1\) and by Lemma \ref{minimalfaithful}, \(G\) has a faithful character with codegree at most \(p^a\), giving \(|G|\le p^{2a-1}\). The quotient \(G/Z_{a-2}\) has class 3 and therefore has the same set of codegrees as \(G\). Let \(\chi\) be an irreducible character of \(G/Z_{a-2}\) such that \(\cod{\chi}=p^a\). By Lemma \ref{nonlinear}, \(\chi\) is nonlinear so we have 
\(p^{a+1}\le \chi(1)p^a = |G:\ker{\chi}| \le |G:Z_{a-2}| \le p^{a+1}\), which shows that \(\chi\) is faithful, \(\chi(1)=p\) and by Lemma \ref{faithfulp}, \(\cd{G/Z_{a-2}}=\{1,p\}\). Let \(\mu \) be an irreducible character of \(G/Z_{a-2}\) with codegree \(p^{a-1}\). Then 
\(p^a=\mu(1)\cod{\mu}=|G:\ker{\mu}|\le |G:Z_{a-2}|=p^{a+1}\), which shows that \(|\ker{\mu}:Z_{a-2}|=p\) and hence \(\ker{\mu}\le Z_{a-1}\).

Suppose \(|Z_{a-1}:Z_{a-2}|=p\). Then \(\mu\) is a faithful character of \(G/Z_{a-1}\), and \(Z_a/Z_{a-1}\) is cyclic, so \(Z_{a-1}\le G'\le Z_{a}\). Since \(p^2\notin \cod{G}\), we have that \(G/G'\) is elementary abelian and hence \(|Z_a:G'|\le p\). Notice that \(Z(\mu)=Z_a\), and by Lemma 2.31 of \cite{thebook}, \(|G:Z(\mu)|=p^2\). We also have \(|G':Z_{a-1}|=p\), as \(G'\ge Z_{a-1}\ge G_3\) and \(G'/G_3\) is elementary abelian while \(G'/Z_{a-1}\) is cyclic. If \(G'=Z_a\), then \(|G/Z_{a-2}|=p^4\), which contradicts that \(\chi\) is a faithful nonlinear character with codegree \(p^a\ge p^4\). Thus \(|Z_a:G'|=p\) and we have \(|G/Z_{a-2}|=p^5\). By Lemma 2.6 of \cite{mepsquared}, this contradicts \(p^2\notin \cod{G}\).

We may now assume that \(|Z_{a-1}:Z_{a-2}|>p\). Let \(G/Z_{a-2}=H\), so \(H\) has class 3, \(|Z(H)|>p\), and \(\mu\) is an irreducible character of \(H\) with \(|\ker{\mu}|=p\). Since \(\chi\) is a faithful irreducible character of \(H\), \(Z(H)\) is cyclic and hence \(\ker{\mu}\) is the unique subgroup of \(H\) of order \(p\). Notice that \(H_3\) is elementary abelian, and since it is contained in the cyclic subgroup \(Z(H)\), we have \(H_3=\ker{\mu}\). As \(H/H_3\) has class 2, \(H'/H_3\le Z(H/H_3)=Z(\mu)/H_3\), which shows that \(H/Z(\mu)\) is abelian and hence \(|H:Z(\mu)|=p^2\). We then have that \(Z(\mu)=Z_2(H)\), as \(Z(\mu)\le Z_2(H)\). By Lemma 2.27 of \cite{thebook}, \(Z(\mu)/\ker{\mu}=Z_2(H)/H_3\) is cyclic. We also have \(H_3\le H'\le Z_2(H)\), so \(H'/H_3\) is cyclic as well as elementary abelian, and hence \(|H':H_3|=p\). This is impossible since \(|Z(H):H_3|=p\) and \(|H':Z(H)|\ge p \). Thus \(c(G)\le a\).
\end{proof}

The next lemma can be used to further shorten our proofs bounding the nilpotence class when a group has exactly four codegrees. With this lemma, we can assume that such a group with nilpotence class greater than 4 has at least one faithful irreducible character, and the codegree of any such character must be as large as possible.

\begin{lem}
\label{4cods_faithfulchi_pa} Let \(G\) be a finite \(p\)-group with \(\cod{G}=\{1,p,p^b,p^a\}\) where \(2\le b<a\). If \(G\) has nilpotence class \(5\), then faithful irreducible characters of \(G\) have codegree \(p^a\).
\end{lem}

\begin{proof}
Let \(\chi\in\irr{G}\) be faithful and suppose \(\cod{\chi}=p^b\).   If \(b=2\), then \(|G|\le p^5\) by Lemma \ref{codorder}, which is impossible since \(G\) has nilpotence class 5. Hence \(b\ge 3\) and \(p^2\notin\cod{G}\). Let \(\varphi \in \irr{G/Z_3}\) be nonlinear, and let \(\cod{\varphi}=p^r\). Note that since \(\varphi\) is nonlinear, \(r\) is equal to either \(b\) or \(a\). Put \(\ker{\varphi}=K\) and \(Z(\varphi)=Y\). By definition, \(\varphi(1)p^r=|G:K|\), and since \(G/Z_3\) has nilpotence class 2, \(|G:Y|=\varphi(1)^2\) by Theorem 2.31 of \cite{thebook}. Combining these equations yields \(p^r\varphi(1)=\varphi(1)^2|Y:K|\), and hence \(p^r=\varphi(1)|Y:K|\). Now \(|G|=\chi(1)p^b\le \chi(1)p^r=\chi(1)\varphi(1)|Y:K|\), so \(\chi(1)\varphi(1)\ge |G:Y||K|=\varphi(1)^2|K|\). This shows that \(p^r> \chi(1)\ge \varphi(1)|K|\), so \(\varphi(1)|K|<p^r=\varphi(1)|Y:K|\). Finally, this implies \(|K|<|Y:K|\). 

Since \(G/K\) has class 2, \(G/Y\) is abelian, and hence \(G'\le Y\). Thus \(Y\ge G'K\ge K\), and since \(Y/K\) is cyclic, the quotient \(Y/G'K\) is also cyclic.  Recalling that \(G/G'\) is elementary abelian since \(p^2\notin \cod{G}\), this also shows that \(Y/G'K\) is elementary abelian. Being both cyclic and elementary abelian, the order of \(Y/G'K\) is at most \(p\). Using the fact that \(Y\) contains \(G'\), we have \(G_3=[G',G]\le [Y,G]\le K\). This shows that \(G'\cap K\) contains \(G_3\). Since \(G'/G_3\) is elementary abelian, \(G'/G'\cap K\) is also elementary abelian. By the Diamond Isomorphism Theorem, \(G'/G'\cap K\) is isomorphic to \(G'K/K\), which is then elementary abelian, but also cyclic, since \(Y/K\) is cyclic. This shows that the order of \(G'K/K\) is at most \(p\). We now have \(|Y:K|=|Y:G'K||G'K:K|\le p^2\). Since \(\varphi\) is an irreducible character of \(G/Z_3\), we know that \(p^3\le |Z_3|\le |K|\), but the previous paragraph showed that \(|K|<|Y:K|\), which is at most \(p^2\), a contradiction. Hence, any faithful characters of \(G\) must have codegree \(p^a\).  \end{proof}

Recall for a \(p\)-group with exactly two character degrees that there is no bound on the nilpotence class of \(G\) when \(\cd{G}=\{1,p\}\), and otherwise \(G\) has  nilpotence class at most \(p\). If the group also has exactly four codegrees, then it has nilpotence class at most 4. 


\begin{proof}[Proof of Theorem \ref{tiered class4} \ref{class4_cd2}]
Let \(G\) be a minimal counterexample. By Lemmas \ref{minimalfaithful} and \ref{4cods_faithfulchi_pa}, \(G\) has at least one faithful irreducible character \(\chi\), and such a character must have codegree \(p^a\). Let \(|G|=p^n\), and notice that \(\chi(1)=p^{n-a}\). Since \(c(G/Z)=4\), Theorem 1.2 of \cite{codandnil} implies that \(\cod{G/Z}=\cod{G}\). Let \(\varphi\in \irr{G/Z}\) have codegree \(p^a\). If \(\varphi\) is nonlinear, then \(\varphi(1)=p^{n-a}\), and hence \(|G:\ker{\varphi}|=\cod{\varphi}\varphi(1)=p^ap^{n-a}=p^n=|G|\). This shows that \(\varphi\) is a faithful irreducible character of \(G\), contradicting that the kernel of \(\varphi\) contains \(Z\). Hence \(\varphi\) must be linear. By Lemma \ref{nonlinear}, \(p^2\in\cod{G}\), and by a similar argument, \(p^a=p^3\). By Lemma \ref{codorder}, we have \(|G|<\cod{\chi}^2=p^6\), implying that \(|G|\le p^5\), and hence \(G\) has nilpotence class at most 4. \end{proof}


\begin{proof}[Proof of Theorem \ref{tiered class4} \ref{class4_cd3}]
Let \(G\) be a minimal counterexample. Notice that for the kernel \(K\) of an irreducible character of \(G\), if \(\cod{G/K}=\cod{G}\) then \(\cd{G/K}=\cd{G}\), for otherwise we have a contradiction with Theorem \ref{tiered class4} \ref{class4_cd2}. Since \(|G|\) is minimal, we must have \(c(G/K)<c(G)\), and we can apply Lemmas \ref{minimalfaithful} and \ref{4cods_faithfulchi_pa}.  Now \(c(G)=5\) and \(G\) has a faithful irreducible character \(\chi\) with codegree \(p^a\). By Lemma \ref{faithfulp}, \(\chi\) must have degree \(p^2\), for otherwise \(|\cd{G}|
=\{1,p\}\), a contradiction. This implies that \(p^{a+2}=\chi(1)\cod{\chi}=|G|\), and since \(G\) has class 5, we see that \(a\) is at least 4. 

The quotient \(G/Z_2\) has nilpotence class 3, hence \(\cod{G/Z_2}=\cod{G}\) and we can find \(\gamma\in\irr{G/Z_2}\) with codegree \(p^a\). Since \(a\ge4\) and \(|\cod{G}|=4\), either \(p^2\) or \(p^3\) is not in \(\cod{G}\). Hence \(\gamma\) is nonlinear, and we have \(p^{a+1}\le \gamma(1)\cod{\gamma} \le |G:Z_2|\le p^a\), which is impossible. Thus \(G\) can have nilpotence class at most 4. 
\end{proof}

To complete the proof of Theorem \ref{tiered class4}, we will use the following lemma for maximal class \(p\)-groups.

\begin{lem}\cite[Theorem 2.4]{memax}
\label{maxp3}
If \(G\) is a maximal class \(p\)-group such that \(|G|\ge p^4\), then \(p^3\in\cod{G}\).
\end{lem}

When the derived subgroup of a \(p\)-group \(G\) has index \(p^2\) in \(G\), some restrictions are placed on the structure of \(G\), including those in the following lemma. In the proof of Theorem \ref{tiered class4} \ref{2gen_class4}, we use these facts to show that if such a group has exactly four codegrees, then it must have nilpotence class at most 4.

\begin{lem}
\label{derived subgrp index p2}
Let \(G\) be a nonabelian \(p\)-group such that \(|G:G'|=p^2\). Then \(|G:G_3|=p^3\), and \(|G:G_4|\le p^5\).
\end{lem}

\begin{proof}
Since \(|G:G'|=p^2\), \(G\) is two-generated, and we can write \(G=\langle x,y\rangle\).  Then \(G'/G_3=\langle [x,y]\rangle/ G_3\), and since \(G/G'\) is elementary abelian, \(|G':G_3|=p\).  For \(G_3/G_4\), we have \(G_3/G_4=\langle [x,y,x], [x,y,y]\rangle/G_4\), and hence \(|G_3:G_4|=p\) or \(p^2\). \end{proof}


\begin{proof}[Proof of Theorem \ref{tiered class4} \ref{2gen_class4}]
Let \(G\) be a minimal counterexample. Then \(G\) has nilpotence class 5, and by Lemmas \ref{minimalfaithful} and \ref{4cods_faithfulchi_pa}, \(G\) has a faithful character \(\chi\) with codegree \(p^a\). Let \(\varphi\in \irr{G/G_3}\) be nonlinear. As \(G/G_3\) has class 2, \(G/Z(\varphi)\) is abelian and \(|G:Z(\varphi)|=\varphi(1)^2\). Since \(Z(\varphi)\) contains \(G'\) which has index \(p^2\) in \(G\), we see that \(\varphi(1)=p\). Put \(\cod{\varphi}=p^r\). Then \(p^{r+1}=\cod{\varphi}\varphi(1)=|G:\ker{\varphi}|\le |G:G_3|\). By Lemma \ref{derived subgrp index p2}, \(|G:G_3|=p^3\), which implies \(r=2\) and \(\cod{G/G_3}=\{1,p,p^2\}\). 

The nilpotence class of \(G/G_4\) is 3, which implies that \(\cod{G/G_4}=\{1,p,p^2,p^a\}\). Let \(\theta \in\irr{G/G_4}\) be nonlinear with codegree \(p^a\). Then \(p^{a+1}\le \cod{\varphi}\varphi(1)=|G:\ker{\varphi}|\le |G:G_4|\). By Lemma \ref{derived subgrp index p2}, \(|G:G_4|\le p^5\), which implies \(a\le 4\). Recalling that \(G\) has a faithful character with codegree \(p^a\), we have that \(|G|\le p^7\). If \(|G|=p^6\) , then by Lemma \ref{maxp3}, \(a=3\), implying that \(|G|\) is actually at most \(p^5\), which contradicts that \(G\) has nilpotence class 5. Hence \(|G|=p^7\) and \(a=4\).  

We now have \(|G:G'|=p^2\), \(|G:G_3|=p^3\), and \(|G:G_4|=p^4\) or \(p^5\). If \(|G_5|=p^2\), then \(G/G_5\) has maximal class, and by Lemma \ref{maxp3}, \(p^3\in\cod{G/G_5}\), a contradiction. Hence \(|G_5|=p\).  Suppose \(|G:G_4|=p^4\) and let \(N\) be a normal subgroup of \(G\) such that \(G_5<N<G_4\). Then \(G/N\) has order \(p^5\) and class 4, and we can again obtain a contradiction from Lemma \ref{maxp3}. Now let \(|G:G_4|=p^5\) and let \(N\) be a normal subgroup of \(G\) such that \(G_4<N<G_3\). Then \(G/N\) has order \(p^4\) and class 3, and Lemma \ref{maxp3} yields the desired contradiction. Thus, no minimal counterexample exists. 
\end{proof}

This proves Theorem \ref{class4_coclasses} in the case when \(G\) has coclass 1. Notice that coclass 1 is equivalent to having maximal class, and a \(p\)-group with nilpotence 4 that is also maximal class has order \(p^5\). 
  
\begin{lem}
\label{class4_coclass2}
Let \(G\) be a finite \(p\)-group with \(\cod{G}=\{1,p,p^b,p^a\}\) where \(2\le b<a\). If \(G\) has coclass 2, then the nilpotence class of \(G\) is at most 4, and \(|G|\le p^6\).
\end{lem}

\begin{proof}
Let \(G\) be a minimal counterexample and suppose \(c(G)=c>5\). The nilpotence class of \(G/Z\) is \(c-1>4\), so \(G/Z\) must be maximal class since \(|G|\) was minimal with coclass 2, but this contradicts Theorem \ref{tiered class4} \ref{2gen_class4}. Hence we may assume \(G\) has nilpotence class 5. 

Suppose \(G\) has no faithful irreducible character. If \(K\) is the kernel of an irreducible character of \(G\), then either \(c(G/K)\le 4\) or \(c(G/K)=5\) and \(G/K\) has maximal class. Since the latter contradicts Theorem \ref{tiered class4} \ref{2gen_class4}, we must have \(c(G/K)\le 4\) for all irreducible characters of \(G\) which contradicts that \(G\) has class 5. Hence \(G\) has a faithful irreducible character \(\chi\), and by Lemma \ref{4cods_faithfulchi_pa}, \(\chi\) has codegree \(p^a\). 

As \(G\) has class 5 and coclass 2, we see that \(|G|=p^7\). If \(|Z_2|=p^3\), then \(G/Z_2\) has class 3 and order \(p^4\), so it has maximal class, and by Lemma  2.3 of \cite{memax} and Lemma \ref{maxp3}, \(\cod{G}=\{1,p,p^2,p^3\}\). Since \(G\) has a faithful character \(\chi\) with codegree \(p^a=p^3\), by Lemma \ref{codorder}, \(|G|\le p^5\), a contradiction. Hence \(|Z_2|=p^2\) and \(|Z|=p\). By Lemma 2.2 of \cite{mepsquared}, we have \(|Z_4|=p^5\).

Let \(\theta\in\irr{G/Z_2}\) have codegree \(p^a\) and notice that \(\theta\) is nonlinear. Thus \(p^{a+1}\le \theta(1)\cod{\theta}\le |G:Z_2|=p^5\). Since \(a\ge4 \), this shows that \(\theta\) is a faithful character of \(G/Z_2\) and has degree \(p\). By Lemma \ref{faithfulp}, \(\cd{G/Z_2}=\{1,p\}\). Now Lemmas \ref{cd 1p} and \ref{1stlem}  imply that either \(|G:Z_3|=p^3\) or \(p^5=|G:Z_2|=p|Z_3:Z_2||G':Z_2|\). Suppose \(|G:Z_3|\ne p^3\). Then \(|Z_3|=p^3\) and \(|G':Z_2|=p^3\), contradicting Theorem \ref{tiered class4} \ref{2gen_class4}. Hence \(|G:Z_3|=p^3\). By \cite{contribution}, Theorem 2.47, \(G'>Z_3\), which forces \(G'\) to have order \(p^5\). This contradicts Theorem \ref{tiered class4} \ref{2gen_class4}, and hence \(G\) must have nilpotence class at most 4. Let \(G\) have order \(p^n\) and nilpotence class \(c\le 4\). Since \(G\) has coclass \(2\), we have \(n-2=c\le 4\), and hence \(|G|\le p^6\). 
\end{proof}

This proves Theorem \ref{class4_coclasses} in the case when \(G\) has coclass 2. The next lemma completes the proof of Theorem \ref{class4_coclasses}.

\begin{lem}
\label{class4_coclass3}
Let \(G\) be a finite \(p\)-group with \(\cod{G}=\{1,p,p^b,p^a\}\) where \(2\le b<a\). If \(G\) has coclass 3, then the nilpotence class of \(G\) is at most 4, and \(|G|\le p^7\).
\end{lem}

\begin{proof}
Let \(G\) be a minimal counterexample and suppose \(c(G)=c>5\). The nilpotence class of \(G/Z\) is \(c-1>4\), so \(G/Z\) must be maximal class or have coclass 2 since \(|G|\) was minimal with coclass 3, but this contradicts Theorem \ref{tiered class4} \ref{2gen_class4} and Lemma \ref{class4_coclass2}. Hence we may assume \(G\) has nilpotence class 5. 

Suppose \(G\) has no faithful irreducible character. If \(K\) is the kernel of an irreducible character of \(G\), then either \(c(G/K)\le 4\), or \(c(G/K)=5\) and \(G/K\) has maximal class or coclass 2, contradicting Theorem \ref{tiered class4} \ref{2gen_class4} and Lemma  \ref{class4_coclass2}, respectively. Hence we must have \(c(G/K)\le 4\) for all irreducible characters of \(G\), which contradicts that \(G\) has class 5. Thus \(G\) has a faithful irreducible character \(\chi\), and by Lemma \ref{4cods_faithfulchi_pa}, \(\chi\) has codegree \(p^a\).

As \(G\) has coclass 3 and nilpotence class 5, we have \(|G|=p^8\), and by Lemma \ref{codorder}, \(5\le a\le 7\). If \(a=7\), then \(\chi(1)=p\). By Lemma \ref{faithfulp} and Theorem \ref{tiered class4} \ref{class4_cd2}, this is impossible, so \(a\) is at most 6. Theorem 1.1 of \cite{mepsquared} implies that \(\cod{G}=\{1,p,p^2,p^a\}\). 

Suppose \(p^2\le |Z|\le p^3\). Let \(\theta\in\irr{G/Z}\) have codegree \(p^a\). Since \(a\) is at least 5, \(\theta\) is nonlinear, and \(p^{a+1}\le \theta(1)\cod{\theta}\le |G:Z|\le p^6\). This shows that \(a=5\) and \(|Z|=p^2\). By \cite{contribution}, Theorem 2.47, \(G_4>Z\) and hence \(|G:G_4|\le p^5\). Since \(c(G/G_4)=3\), \(\cod{G/G_4}=\{1,p,p^2,p^5\}\), which is impossible. Hence \(Z=G_5\) has order \(p\). 

 The quotient \(G/Z_2\) has class 3 and hence \(p^a\in \cod{G/Z_2}\). As \(|G:Z_2|\le p^6\) and \(p^a=p^5\) or \(p^6\) is the codegree of a nonlinear character of \(G/Z_2\), we must have \(|G:Z_2|=p^6\) and \(a=5\). Let \(\gamma\in\irr{G/Z_2}\) have codegree \(p^5\). Then \(\gamma\) is nonlinear, so \(p^6\le \gamma(1)\cod(\gamma)\le |G:Z_2|=p^6\), which shows that \(\gamma\) is a faithful character of \(G/Z_2\) with degree \(p\). By Lemma \ref{faithfulp}, this implies that \(\cd{G/Z_2}=\{1,p\}\).  

Theorem 2.47 of \cite{contribution} implies that \(G'> Z_3\), and by Theorem \ref{tiered class4} \ref{2gen_class4} we may assume \(|G:G'|\ge p^3\). Hence \(p^4\le |G'|\le p^5\) and \(|Z_3|\le p^4\). Since \(|G:Z_3|\ge p^4\), Lemma \ref{cd 1p} implies that \(G/Z_2\) has an abelian subgroup of index \(p\). By Lemma \ref{1stlem}, we have \(p^6=|G/Z_2|=p|Z_3/Z_2||G'/Z_2|\) which shows that \(|G'|=p^5\) and \(|Z_3|=p^4\). As \(G/Z_2\) has an abelian subgroup of index \(p\), \(G/Z_3\) also has such a subgroup \(A/Z_3\), and Lemma 12.12 of \cite{thebook} states that \(p^3=|A/Z_3|=|G'/Z_3||(A/Z_3)\cap (Z_4/Z_3)|\). Hence \(|(A/Z_3)\cap (Z_4/Z_3)|=p^2\), which shows that \(|Z_4|=p^6\).

Let \(x\in G'\). Then \(z\in Z_4\) implies that \([x,g]\in Z_3\) for all \(g\in G\). Since \(|G':Z_3|=p\), \(x^p\in Z_3\), and hence passing to \(G/Z_2\), we have \(\overline{1}=\overline{[x^p,g]}=\overline{[x,g]^p}\). This shows that \([x,g]^p\in Z_2\) for all \(x\in G'\) and \(g\in G\), and we conclude that \(G_3/Z_2\) is elementary abelian. Recall that \(G/Z_2\) has a faithful character, and hence \(Z_3/Z_2\) is cyclic. Thus \(|G_3|=p^3\). 

Let \(Z_3=\langle a,Z_2\rangle\) and observe that \(G_3=\langle a^p, Z_2\rangle\). Then \([a,g]\in Z_2\), and passing to \(G/Z\) we have \(\overline{[a^p,g]}=\overline{[a,g]^p}=\overline{1}\), which shows that \(\overline{a^p}\) commutes with all \(\overline{g}\in G/Z\), and hence \(a^p\in Z_2\), a contradiction since \(G_3\ne Z_2\).  Therefore \(G\) must have nilpotence class at most 4. Since \(G\) has coclass \(3\), we have \(n-3=c\le 4\), and hence \(|G|\le p^7\). 
\end{proof}

For the remainder of this section we will consider only those groups which satisfy Hypothesis \((\ast)\). Together, Lemmas \ref{star_class4_coclass4} - \ref{star_class4_coclass7} prove Theorem \ref{star_class4_coclasses}.

\begin{lem}
\label{star_class4_coclass4}
Let \(G\) be a finite \(p\)-group with \(\cod{G}=\{1,p,p^b,p^a\}\) where \(2\le b<a\). If \(G\) has coclass 4, and \(G\) and all of its quotients satisfy Hypothesis \((\ast)\), then the nilpotence class of \(G\) is at most 4 and \(|G|\le p^8\).
\end{lem}

\begin{proof}
Let \(G\) satisfy \((\ast)\) and let \(|G|\) be minimal such that \(G\) has coclass 4 and \(|\cod{G}|=4\). Suppose \(c(G)=c>5\). The nilpotence class of \(G/Z\) is \(c-1>4\), so \(G/Z\) has coclass at most 3 since \(|G|\) was minimal with coclass 4, but this contradicts Theorem \ref{tiered class4} \ref{2gen_class4}, Lemma \ref{class4_coclass2}, and Lemma \ref{class4_coclass3}. Hence we may assume \(G\) has nilpotence class 5. 

Suppose \(G\) has no faithful irreducible character. If \(K\) is the kernel of an irreducible character of \(G\), then either \(c(G/K)\le 4\), or \(c(G/K)=5\) and \(G/K\) has coclass at most 3, contradicting Theorem \ref{tiered class4} \ref{2gen_class4}, Lemma \ref{class4_coclass2}, and Lemma \ref{class4_coclass3}. Hence we must have \(c(G/K)\le 4\) for all irreducible characters of \(G\), which contradicts that \(G\) has class 5. Thus \(G\) has a faithful irreducible character \(\chi\), and by Lemma \ref{4cods_faithfulchi_pa}, \(\chi\) has codegree \(p^a\). By Lemma  \ref{codorder}, \(a\ge 5\). As \(c(G/Z_2)=3\), we have \(p^a\in\cod{G/Z_2}\), and since \(G\) satisfies \((\ast)\), the order of \(G/Z_2\) is at most \(p^6\), which shows that \(a\le 5\). Hence \(a=5\) which forces \(Z_2\) to have order \(p^3\).

Let \(\gamma\in\irr{G/Z_2}\) have codegree \(p^5\) and notice that \(\gamma\) is a faithful character of \(G/Z_2\) with degree \(p\). By Theorem \ref{tiered class4}  \ref{2gen_class4}, \(G'\) has order at most \(p^6\), and \(G'>Z_3\) by \cite{contribution}, Theorem 2.47. Thus Lemmas \ref{faithfulp} and \ref{1stlem} imply that \(|G/Z_2|=p^6=p|G'/Z_2||Z_3/Z_2|\). This shows that \(|G'|=p^6\) and \(|Z_3|=p^5\).

Suppose \(Z_2/Z\) has exponent greater than \(p\). Then there exists \(x\in Z_2\) such that \(x^p\notin Z\), and \(g\in G\) such that \([x^p,g]\ne 1\). Note that this also implies \(|Z|=p\). Since \([x,g]\in Z\), we have \([x^p,g]=[x,g]^p=1\), as \(|Z|=p\), a contradiction. Hence \(Z_2/Z\) is elementary abelian. Repeating this argument in the quotient group \(G/Z\), we conclude that \(Z_3/Z_2\) is also elementary abelian, which contradicts that the center of \(G/Z_2\), \(Z_3/Z_2\), is cyclic with order \(p^2\). Hence \(G\) must have nilpotence class at most 4. Since \(G\) has coclass \(4\), we have \(n-4=c\le 4\), and hence \(|G|\le p^8\).  \end{proof}

\begin{lem}
\label{star_class4_coclass5}
Let \(G\) be a finite \(p\)-group with \(\cod{G}=\{1,p,p^b,p^a\}\) where \(2\le b<a\). If \(G\) has coclass 5, and \(G\) and all of its quotients satisfy Hypothesis \((\ast)\), then the nilpotence class of \(G\) is at most 4 and \(|G|\le p^9\).
\end{lem}

\begin{proof}
Let \(G\) satisfy \((\ast)\) and let \(|G|\) be minimal such that \(G\) has coclass 5 and \(|\cod{G}|=4\). Suppose \(c(G)=c>5\). The nilpotence class of \(G/Z\) is \(c-1>4\), so \(G/Z\) has coclass at most 4 since \(|G|\) was minimal with coclass 5, but this contradicts Theorem \ref{tiered class4} \ref{2gen_class4}, Lemma \ref{class4_coclass2}, Lemma \ref{class4_coclass3}, and Lemma \ref{star_class4_coclass4}. Hence we may assume \(G\) has nilpotence class 5. 

Suppose \(G\) has no faithful irreducible character. If \(K\) is the kernel of an irreducible character of \(G\), then either \(c(G/K)\le 4\), or \(c(G/K)=5\) and \(G/K\) has coclass at most 4, contradicting Theorem \ref{tiered class4}  \ref{2gen_class4}, Lemma \ref{class4_coclass2}, Lemma \ref{class4_coclass3}, and Lemma \ref{star_class4_coclass4}. Hence we must have \(c(G/K)\le 4\) for all irreducible characters of \(G\), which contradicts that \(G\) has class 5. Thus \(G\) has a faithful irreducible character \(\chi\), and by Lemma \ref{4cods_faithfulchi_pa}, \(\chi\) has codegree \(p^a\). Since \(G\) has class 5 and coclass 5, we have \(|G|=p^{10}\), and by Lemma \ref{codorder}, \(a\ge 6\). As \(c(G/Z_2)=3\), we have \(p^a\in\cod{G/Z_2}\), and since \(G\) satisfies \((\ast)\), the order of \(G/Z_2\) is at most \(p^7\), which shows that \(a\le 6\). Hence \(a=6\) which forces \(Z_2\) to have order \(p^3\).

Let \(\gamma\in\irr{G/Z_2}\) have codegree \(p^6\) and notice that \(\gamma\) is a faithful character of \(G/Z_2\) with degree \(p\). By Theorem \ref{tiered class4}  \ref{2gen_class4}, \(G'\) has order at most \(p^7\), and \(G'>Z_3\) by \cite{contribution}, Theorem 2.47. Thus Lemmas \ref{faithfulp} and \ref{1stlem} imply that \(|G/Z_2|=p^7=p|G'/Z_2||Z_3/Z_2|\). This shows that \(|G'|=p^7\) and \(|Z_3|=p^5\).

Suppose \(Z_2/Z\) has exponent greater than \(p\). Then there exists \(x\in Z_2\) such that \(x^p\notin Z\), and \(g\in G\) such that \([x^p,g]\ne 1\). Note that this also implies \(|Z|=p\). Since \([x,g]\in Z\), we have \([x^p,g]=[x,g]^p=1\), as \(|Z|=p\), a contradiction. Hence \(Z_2/Z\) is elementary abelian. Repeating this argument in the quotient group \(G/Z\), we conclude that \(Z_3/Z_2\) is also elementary abelian, which contradicts that the center of \(G/Z_2\), \(Z_3/Z_2\), is cyclic with order \(p^2\). Hence \(G\) can have nilpotence class at most 4. Since \(G\) has coclass \(5\), we have \(n-5=c\le 4\), and hence \(|G|\le p^9\). 
\end{proof}

\begin{lem}
\label{star_class4_coclass6}
Let \(G\) be a finite \(p\)-group with \(\cod{G}=\{1,p,p^b,p^a\}\) where \(2\le b<a\). If \(G\) has coclass 6, and \(G\) and all of its quotients satisfy Hypothesis \((\ast)\), then the nilpotence class of \(G\) is at most 4 and \(|G|\le p^{10}\).
\end{lem}

\begin{proof}
Let \(G\) satisfy \((\ast)\) and let \(|G|\) be minimal such that \(G\) has coclass 6 and \(|\cod{G}|=4\). Suppose \(c(G)=c>5\). The nilpotence class of \(G/Z\) is \(c-1>4\), so \(G/Z\) has coclass at most 5 since \(|G|\) was minimal with coclass 6, but this contradicts Theorem \ref{tiered class4} \ref{2gen_class4} and Lemmas \ref{class4_coclass2} - \ref{star_class4_coclass5}. Hence we may assume \(G\) has nilpotence class 5. 

Suppose \(G\) has no faithful irreducible character. If \(K\) is the kernel of an irreducible character of \(G\), then either \(c(G/K)\le 4\), or \(c(G/K)=5\) and \(G/K\) has coclass at most 5, contradicting Theorem \ref{tiered class4} \ref{2gen_class4} and Lemmas \ref{class4_coclass2} - \ref{star_class4_coclass5}. Hence we must have \(c(G/K)\le 4\) for all irreducible characters of \(G\), which contradicts that \(G\) has class 5. Thus \(G\) has a faithful irreducible character \(\chi\), and by Lemma \ref{4cods_faithfulchi_pa}, \(\chi\) has codegree \(p^a\). Since \(G\) has class 5 and coclass 6, we have \(|G|=p^{11}\), and by Lemma  \ref{codorder}, \(a\ge 6\). As \(c(G/Z_2)=3\), we have \(p^a\in\cod{G/Z_2}\), and since \(G\) satisfies \((\ast)\), the order of \(G/Z_2\) is at most \(p^8\), which shows that \(a\le 7\).

Suppose \(a=7\). Then \(|Z_2|=p^3\). Let \(\gamma\in\irr{G/Z_2}\) have codegree \(p^7\) and notice that \(\gamma\) is a faithful character of \(G/Z_2\) with degree \(p\). By Theorem \ref{tiered class4}   \ref{2gen_class4}, \(G'\) has order at most \(p^7\), and \(G'>Z_3\) by \cite{contribution}, Theorem 2.47. Thus Lemmas \ref{faithfulp} and \ref{1stlem} imply that \(|G/Z_2|=p^8=p|G'/Z_2||Z_3/Z_2|\). This shows that \(Z_3/Z_2\) has order at least \(p^2\), and since \(\gamma\) is faithful, \(Z_3/Z_2\) is cyclic. 

Suppose \(Z_2/Z\) has exponent greater than \(p\). Then there exists \(x\in Z_2\) such that \(x^p\notin Z\), and \(g\in G\) such that \([x^p,g]\ne 1\). Note that this also implies \(|Z|=p\). Since \([x,g]\in Z\), we have \([x^p,g]=[x,g]^p=1\), as \(|Z|=p\), a contradiction. Hence \(Z_2/Z\) is elementary abelian. Repeating this argument in the quotient group \(G/Z\), we conclude that \(Z_3/Z_2\) is also elementary abelian, which contradicts that \(Z_3/Z_2\) is cyclic with order at least \(p^2\). Hence \(a=6\). 

Let \(\chi\in\irr{G}\) be faithful with codegree \(p^6\). Then \(\chi(1)=p^5\), so \(|G:Z|\ge \chi(1)^2\) implies \(|Z|=p\). As before, this shows that \(Z_2/Z\) is elementary abelian, and similarly, \(Z_i/Z_{i-1}\) is elementary abelian for \(1\le i\le 5\). Since \(p^6\in\cod{G/Z_2}\), we have \(|G:Z_2|\ge p^7\), so \(|Z_2|=p^3\) or \(p^4\). If \(|Z_2|=p^4\), then \(G/Z_2\) has a faithful character of degree \(p\). As in the previous lemmas, \(|G:G'|\ge p^3\) and \(G'>Z_3\). Hence Lemmas \ref{faithfulp} and \ref{1stlem} imply that \(|G/Z_2|=p^7=p|G'/Z_2||Z_3/Z_2|\). This shows that \(Z_3/Z_2\) has order at least \(p^2\), and since \(\gamma\) is faithful, \(Z_3/Z_2\) is cyclic. This is impossible since \(Z_3/Z_2\) is elementary abelian. Hence \(|Z_2|=p^3\). 

Suppose \(|Z_3|=p^4\). Then \(|Z|=p\), for otherwise \(G/Z\) contradicts Hypothesis \((\ast )\). Let \(N\lhd G\) such that \(Z<N<Z_2\). Put \(Z(G/N)=X/N\) and notice that \(Z_2\le X \le Z_3\). If \(X=Z_2\), then \(Z(G/N)=Z_2/N\) and \(Z_2(G/N)=Z_3/_N\). Since \(G/N\) has class 3 or 4, this contradicts Hypothesis \((\ast )\). Hence \(X=Z_3\). Now \(Z_2(G/N)=Z_4/N\), so \(G/N\) has class 3, which implies \(N\ge G_4\). Since \(G_4>Z\) by Theorem 2.47 of \cite{contribution}, we have \(p^2\le |G_4|\le |N|=p^2\), which implies \(N=G_4\) for all normal subgroups of G lying between \(Z\) and \(Z_2\). This shows that \(Z_2/Z\) is cyclic, contradicting that it is elementary abelian of order \(p^2\). Hence \(|Z_3|\ge p^5\). 

Now let \(H\lhd G\) such that \(Z_2<H<Z_3\) and assume \(H\ne G_3\). Then \(G/N\) has class 3, and \(p^6\in\cod{G/N}\). Since \(|G:N|=p^7\), there exists a faithful character of \(G/N\) with degree \(p\). Put \(Y/N=Z(G/N)\) and notice that \(Y/N\) is cyclic. The center of \(G/N\) lies between \(Z_3\) and \(Z_4\), and since \(Z_4/Z_3\) is elementary abelian, this shows that \(|Y:N|\) is at most \(p^2\) and hence \(|G:Y|\ge p^5\). By  Lemmas \ref{cd 1p} and \ref{1stlem} we have \(|G/N|=p^7=p|G'/N||Y/N|\le p^3|G'/N|\), which forces \(G'\) to have order \(p^8\) and \(Y\) to have order \(p^6\). Since \(Y/N\) is cyclic while \(Z_3/N\) is elementary abelian, \(Z_3\) must have order \(p^5\). Lemma \ref{cd 1p} implies that \(G/N\) has an abelian subgroup of index \(p\), and hence \(G/Z_3\) must also have such a subgroup, so we can apply Lemma \ref{1stlem} to \(G/Z_3\). Now \(p^6=p|G'/Z_3||Z_4/Z_3|=p^4|Z_4/Z_3|\), which shows that \(Z_4\) has order \(p^7\), which is impossible since \(G'\) is contained in \(Z_4\) and has order \(p^8\). Thus \(G\) must have class at most 4. Since \(G\) has coclass \(6\), we have \(n-6=c\le 4\), and hence \(|G|\le p^{10}\).  \end{proof}

\begin{lem}
\label{star_class4_coclass7}
Let \(G\) be a finite \(p\)-group with \(\cod{G}=\{1,p,p^b,p^a\}\) where \(2\le b<a\). If \(G\) has coclass 7, and \(G\) and all of its quotients satisfy Hypothesis \((\ast)\), then the nilpotence class of \(G\) is at most 4 and \(|G|\le p^{11}\).
\end{lem}

\begin{proof}
Let \(G\) satisfy \((\ast)\) and let \(|G|\) be minimal such that \(G\) has coclass 7 and \(|\cod{G}|=4\). Similarly to previous lemmas, we may assume \(G\) has class 5 and a faithful irreducible character \(\chi\) with codegree \(p^a\). The order of \(G\) is \(p^{12}\), so \(a\ge 7\) by Lemma \ref{codorder}. As \(G/Z_2\) has order at most \(p^9\), we also have \(a\le 8\). 

Suppose first that \(a=8\), and hence \(|G:Z_2|=p^9\). Then \(G/Z_2\) has a faithful irreducible character of degree \(p\), and by Lemma \ref{faithfulp}, \(\cd{G/Z_2}=\{1,p\}\). Suppose \(Z_2/Z\) has exponent greater than \(p\). Then there exists \(x\in Z_2\) such that \(x^p\notin Z\), and \(g\in G\) such that \([x^p,g]\ne 1\). Note that this also implies \(|Z|=p\). Since \([x,g]\in Z\), we have \([x^p,g]=[x,g]^p=1\), as \(|Z|=p\), a contradiction. Hence \(Z_2/Z\) is elementary abelian. Repeating this argument in the quotient group \(G/Z\), we conclude that \(Z_3/Z_2\) is also elementary abelian, and since it is also cyclic, we have \(|Z_3|=p^4\). By Lemma \ref{1stlem}, \(p^9=p|G'/Z_2||Z_3/Z_2|\), and hence \(|G'|=p^{10}\). This contradicts Theorem \ref{tiered class4} \ref{2gen_class4}, so \(a=8\) is impossible. 

Now \(a=7\) and \(|G:Z_2|=p^8\) or \(p^9\). Suppose \(|G:Z_2|=p^8\). Then \(G/Z_2\) has a faithful irreducible character of degree \(p\), so \(\cd{G/Z_2}={1,p}\) and \(Z_3/Z_2\) is cyclic. The order of \(Z_4\) is at least \(p^7\), for otherwise \(G/Z_2\) would violate Hypothesis \((\ast)\). If \(|Z_3|=p^5\), then by Lemma \ref{1stlem}, \(p^8=p|G'/Z_2||Z_3/Z_2|\), which implies that \(|G'|=p^{10}\), contradicting Theorem \ref{tiered class4}  \ref{2gen_class4}. If \(|Z_3|=p^6\), we have \(|G'|=p^9\). Applying Lemma \ref{1stlem} to \(G/Z_3\), this would imply \(|Z_4|=p^8\), which is impossible since \(Z_4\) contains \(G'\). Thus \(|Z_3|\ge p^7\), so \(|G/Z_2|=p^8=p|Z_3/Z_2||G'/Z_2|\ge p^4|G'/Z_2|\), which implies that \(|G'|\le p^8\). Since \(G'>Z_3\) by Theorem 2.47 of \cite{contribution}, this shows that \(|G'|=p^8\) and \(|Z_3|=p^7\).  Applying Lemma \ref{1stlem} to \(G/Z_3\), we obtain \(|Z_4|=p^{10}\). 

Notice that \(G\) has a faithful irreducible character \(\chi\) with codegree \(p^7\), so \(\chi(1)=p^5\), and hence \(|Z|\le p^2\). Suppose \(|Z|=p\). Put \(Z_3=\langle x,Z_2\rangle\), and let \(\overline{G}=G/Z\). Since \(x^{p^2}\notin Z_2\), there exists some \(g\in G\) such that \(\overline{1}\not\equiv \overline{[x^{p^2},g]}\equiv \overline{[x,g]^{p^2}} \equiv \overline{1}\), since \(\overline{[x,g]}\in Z_2/Z\), which has order \(p^2\). Hence \(|Z|=p\). Suppose \(Z_2/Z\) has exponent greater than \(p\). Then there exists \(y\in Z_2\) such that \(y^2\notin Z\), and \(g\in G\) such that \(1\neq [y^p,g]=[y,g]^p=1\), a contradiction. Hence \(Z_2/Z\) is elementary abelian. Repeating this argument in \(G/Z\), we see that \(Z_3/Z_2\) is also elementary abelian, which is impossible since \(Z_3/Z_2\) is cyclic of order \(p^3\). Thus \(Z_2\) must have order \(p^3\). 

Suppose \(Z_2/Z\) has exponent greater than \(p\). Then there exists \(x\in Z_2\) such that \(x^p\notin Z\), and \(g\in G\) such that \([x^p,g]\ne 1\). Note that this also implies \(|Z|=p\). Since \([x,g]\in Z\), we have \([x^p,g]=[x,g]^p=1\), as \(|Z|=p\), a contradiction. Hence \(Z_2/Z\) is elementary abelian, and similarly, \(Z_i/Z_{i-1}\) is elementary abelian for \(1\le i\le 5\). 

Suppose \(|Z_3|\ge p^5\). Let \(N\lhd G\) such that \(|N|=p^4\), \(Z_2<N<Z_3\), and \(c(G/N)=3\). Put \(X/N=Z(G/N)\), and notice that \(Z_3\le X\le Z_4\). Since \(G/N\) has class 3, \(p^7\in \cod{G/N}\), and hence \(G/N\) has a faithful irreducible character of degree \(p\). Thus \(X/N\) is cyclic, and as \(Z_3/N\) and \(Z_4/Z_3\) are elementary abelian, we have \(|X:N|\le p^2\). By Lemmas \ref{faithfulp} and \ref{1stlem}, \(|G/N|=p^8=p|G'/N||X/N|\le p^3|G'/N|\). Thus \(|G'|=p^9\). If \(|Z_3|=p^5\) or \(p^6\), then by applying Lemma \ref{1stlem} to \(G/Z_3\) we have \(|Z_4|=p^7\) or \(p^8\), which is impossible since \(Z_4\) contains \(G'\). Thus \(|Z_3|\ge p^7\), which implies \(|Z_2:Z_3|\ge p^4\). Since \(|X:N|\le p^2\), and \(X\ge Z_3\), this is impossible. Hence \(|Z_3|=p^4\). 

If \(|Z|=p^2\), then \(G/Z\) would violate Hypothesis \((\ast)\), so we have \(|Z|=p\). Let \(H\lhd G\) such that \(Z<H<Z_2\). and \(c(G/H)=4\). Put \(Z(G/H)=Y/H\) and notice that \(Z_2\le Y\le Z_3\). If \(Y=Z_3\), then \(G/H\) would have class 3, a contradiction, so \(Y=Z_2\). Now \(Z_2(G/H)= Z_3/H\), which contradicts Hypothesis \((\ast)\). Hence \(G\) must have class at most 4. Since \(G\) has coclass \(7\), we have \(n-7=c\le 4\), and hence \(|G|\le p^{11}\). \end{proof}

\end{document}